\documentclass{article}
\usepackage{amsmath,amssymb,amsthm, algorithm, algorithmic, enumerate,listings,marvosym,hyperref,eucal,verbatim,amscd}

\title{The modular Cauchy kernel for the Hilbert modular surface}
\author{Nina Sakharova\thanks{The article was prepared within the framework of a subsidy granted to the HSE by the Government of the Russian Federation for the implementation of the Global Competitiveness Program.  }}
\date{}

\newtheorem{theorem}{Theorem}
\newtheorem{lemma}{Lemma}
\newtheorem{corollary}{Corollary}

\theoremstyle{definition}
\newtheorem{definition}{Definition}

\theoremstyle{remark}
\newtheorem{remark}{Remark}

\usepackage[pdftex]{graphicx}

\usepackage{bm}
\usepackage{amsfonts}
\usepackage{amsmath}
\usepackage{amssymb}
\usepackage{amsthm}

\renewcommand{\pmod}[1]{(\textmd{mod}\hspace{.5mm}#1)}

\begin{document}
\maketitle
\begin{center}
\large{ \textit{National Research University Higher School of Economics, Russian Federation \\ Laboratory of Mirror Symmetry, 6 Usacheva str., Moscow, Russia, 119048.\\
saharnina@gmail.com}}
\end{center}
\begin{abstract}
In this paper we construct the modular Cauchy kernel on the Hilbert modular surface $\Xi_{\mathrm{Hil},m}(z)(z_2-\bar{z_2})$, i.e. the function of two variables, $(z_1, z_2) \in \mathbb{H} \times \mathbb{H}$, which is invariant  under the action of the Hilbert modular group, with the first order pole on the Hirzebruch-Zagier divisors. The derivative of this function with respect to $\bar{z_2}$ is the function $\omega_m (z_1, z_2)$ introduced by Don Zagier in \cite{Za1}. We consider the question of the convergence and the Fourier expansion of the kernel function. The paper generalizes the first part of the results obtained in the preprint \cite{Sa}.
\end{abstract}

\section{Introduction} Let $F=\mathbb{Q}(\sqrt{d})$ is a real quadratic field, $d>1$ is a squarefree integer. The discriminant of $F$ is $$D =
 \begin{cases}
   d & d \equiv 1 \pmod 4\\
   4d &d \equiv 2,3 \pmod 4.
 \end{cases}$$
Let $\mathcal{O}_F$ is a ring of integers of $F$,
$$\mathcal{O}_F =
 \begin{cases}
   \mathbb{Z}+\frac{1+\sqrt{d}}{2}~\mathbb{Z} & d \equiv 1 \pmod 4\\
   \mathbb{Z} +\sqrt{d}~\mathbb{Z}  &d \equiv 2,3 \pmod 4.
 \end{cases}$$
We right  $\mathrm{N}(x)=xx'$ for the norm of  $x$, and  $\mathrm{tr}(x)=x+x'$ for the trace of $x$. Denote by $$\mathfrak{d}_F=(\mathcal{O}_F^{\vee})^{-1}=\left\{x \in F:~~ \mathrm{tr}(x \cdot y) \in \mathbb{Z}, ~\forall y \in  \mathcal{O}_F \right\}$$ the different of $F$. \\

We denote by $\mathcal{B}_m$ the set of matrices $$\mathcal{B}_m=\left\{B=\left(\begin{matrix}
	-b & \lambda'\\
	\lambda & -a\\
	\end{matrix}\right) |~~ \lambda \in \mathfrak{d}^{-1}_F; ~ a, b \in \mathbb{Z},~ \det B = ab-\lambda\lambda'= m/ D \right\}.$$
Let, $$\mu_{\gamma}(z_1, z_2)=az_1z_2+\lambda z_1+\lambda'z_2+b,$$ where $\gamma=\left(\begin{matrix}
	\lambda & b\\
	a & \lambda'\\
	\end{matrix}\right).$ 

 Let $B=\left(\begin{matrix}
	a & \lambda'\\
	\lambda & b\\ \end{matrix}\right)$, $z =(z_1, z_2) \in \mathbb{H}^2$, $M(z)=\left(\begin{matrix}
	z_1z_2 & z_1\\
	z_2 & 1
	\end{matrix}\right)$, $B^{\ast}=\left(\begin{matrix}
	b' & -\lambda\\
	\lambda' & a'\\
	\end{matrix}\right)$, and, according to the notation of \cite{BGZ}, $(M(z), B)= -\mathrm{tr}(M(z) \cdot B^{\ast})$.\\
\newpage
Now we define the modular Cauchy kernel for the Hilbert modular surface: \begin{definition} Let $m \neq 0$.
\begin{multline} \Xi_{\mathrm{Hil},m}(z)(z_2-\bar{z_2})=\frac{1}{2}\sum_{B \in \mathcal{B}_m}\frac{(z_2-\bar{z_2})}{(M(z_1, z_2), B)(M(z_1, \bar{z_2}), B)}  =\frac{1}{2}\sum_{\substack{\gamma=\left(\begin{smallmatrix}
	\lambda & b\\
	a & \lambda'\\
	\end{smallmatrix}\right) \\ab - N(\lambda) =m/D}}\dfrac{(z_2-\bar{z_2})}{\mu_{\gamma}(z_1, z_2)\mu_{\gamma}(z_1, \bar{z_2})}=\\ =\frac{1}{2} \sum_{\substack{(a, b, \lambda) \in L^{\vee}/\left\{\pm1\right\} \\ N(\lambda)-ab =-m/D}}\dfrac{(z_2-\bar{z_2})}{(az_1z_2+\lambda z_1+\lambda'z_2+b)^{2k-1}(az_1\bar{z_2}+\lambda z_1+\lambda'\bar{z_2}+b)^{2k-1}}. \label{Xi}
\end{multline} \end{definition} This function has first-order poles on Hirzebruch-Zagier divisors. \\

This series is ``almost'' modular invariant with respect to the Hilbert modular group (the verification of the invariance is similar to the calculation in \cite{Za1}):
$$\Xi_{\mathrm{Hil},m}\left(  \frac{\alpha z_1+\beta}{\gamma z_1+\delta}, \frac{\alpha' z_2+\beta'}{\gamma' z_2+\delta'} \right)\left( \frac{\alpha' z_2+\beta'}{\gamma' z_2+\delta'}- \frac{\alpha' \bar{z_2}+\beta'}{\gamma' \bar{z_2}+\delta'}\right) =(\gamma z_1+\delta)^2~ \Xi_{\mathrm{Hil},m}(z).$$
The series ~$\Xi_{\mathrm{Hil},m}(z_1,z_2)$~ does not converge absolutely. Following the lead of E. Hecke, we investigate the series: $$\Xi_{\mathrm{Hil},m}(z, s)=\sum_{B \in \mathcal{B}_m} \dfrac{\overline{(M(z_1, z_2), B)}~ \overline{(M(z_1, \bar{z_2}), B)}}{|(M(z_1, z_2), B)|^{2s}|(M(z_1, \bar{z_2}), B)|^{2s}},$$ where $s$ is a complex number.  In \cite{Sa} it was shown that, in the case $D=1, $ this series does not have a pole for $s =1$. A similar result holds for an arbitrary positive discriminant. Here we omit the corresponding calculation. We define $$\Xi_{\mathrm{Hil},m}(z)= \lim_{s \rightarrow 1}~\Xi_{\mathrm{Hil},m}(z, s)$$  and, using such a definition, we will find the Fourier expansion of the function $\Xi_{\mathrm{Hil},m}(z)$.\\

In \cite{Za1}, Don Zagier introduced the series: $$G_a(m, \nu)=\sum_{\substack{\lambda \in \mathfrak{d}^{-1}_F/a\mathcal{O}_F \\ \lambda\lambda' \equiv  -m/D \pmod{a \mathbb{Z}}}}e^{2 \pi i \frac{\mathrm{Tr} (\nu \lambda)}{ a}}.$$ The following lemma is due Zagier:

\begin{lemma} Let $a \in \mathbb{N}$, $m\in \mathbb{Z}$ and $\nu \in \mathfrak{d}^{-1}_F\neq0$. Then there exists a constant $C >0$, such that  $$|G_a(m, \nu)| \leq C d(a) \sqrt{a|\nu\nu'|},$$ where $d(a)$ is a number of positive divisors of  $a$.
\end{lemma} \begin{corollary}\label{est} The series $\dfrac{G_a (m, \nu)} {a^s}$ converges for $\Re s > 3/4$. \end{corollary}

We denote by $$I_{1}(z)= \frac{z}{4\pi i} \int_{c-i \infty}^{c+ i \infty} t^{-2}e^{t+z^2/4t} dt, ~~~~J_{1}(z)= \frac{z}{4\pi i} \int_{c-i \infty}^{c+ i \infty} t^{-2}e^{t-z^2/4t} dt$$  the Bessel functions of the first kind (modified and unmodified, respectively).\\

\begin{theorem}\label{threorem1} Let $p=e^{2\pi iz_1},$ $q=e^{2\pi iz_2}$ and $\widetilde{q}=e^{2\pi i\overline{z_2}}.$ For $\Im z_1 \Im z_2 > m/D$ and $m\neq0$, the function \\ $\Xi_{\mathrm{Hil},m}(z)(z_2-\bar{z_2})$ has the following Fourier expansion:
\begin{multline} \frac{1}{2\pi i}~ \Xi_{\mathrm{Hil},m}(z)(z_2-\bar{z_2})=-\lim_{s \rightarrow 1}\frac{\pi}{\sqrt{D}}\sum_{a>0}\frac{G_a(m, 0)}{a^{2}}\frac{1}{\Gamma(s-1)}\frac{1}{\Im z_1}-\\
-\sum_{\substack{\lambda \in \mathfrak{d}^{-1}_F,~\lambda>0 \\ \lambda\lambda'=-m/D}}~\sum_{r>0}~\frac{1}{\lambda'}\left( e^{2 \pi i r (\lambda z_1+\lambda' z_2)}-e^{2 \pi i r (\lambda z_1+\lambda' \bar{z_2})}\right) \\
- 2\pi  \sum_{\substack{\nu >0, \\ \nu'>0}}\sum_{a>0}^{\infty} \frac{G_a(m, \nu)}{a^2}\sqrt{\frac{\nu}{\nu'm}} J_1\left(  \frac{4 \pi \sqrt{\nu\nu'm/D}}{a}\right)e^{2\pi i(\nu z_1+\nu'z_2)} +\\
+  2\pi  \sum_{\substack{\nu >0, \\ \nu'<0}}\sum_{a>0}^{\infty} \frac{G_a(m, \nu)}{a^2}\sqrt{\frac{\nu}{\nu'm}} I_1\left( \frac{4 \pi \sqrt{\nu\nu'm/D}}{a}\right)e^{2\pi i(\nu z_1+\nu'\bar{z_2})}
\end{multline}
\end{theorem}

If  $N_b(n)= \# \left\{ \lambda \in  \mathcal{O}/b ~|~ N(\lambda)\equiv n \pmod{b}   \right\}$ (see \cite{Za1}), then

\begin{corollary} The zero coefficient of the Fourier series of the function $\Xi_{\mathrm{Hil},m}(z)(z_2-\bar{z_2})$ has the form: $$\lim_{s\rightarrow 1} ~\frac{2\pi^2 i  }{\sqrt{D}}\sum_{a>0}\frac{G_a(m, 0)}{a^{2}}\frac{1}{\Gamma(s-1)}\frac{1}{ \Im z_1}=\lim_{s\rightarrow 1}~\frac{2\pi^2 i }{D^{3/2}}\sum_{a>0}\frac{N_{aD}(m)}{a^2}\frac{1}{\Gamma(s-1)}\frac{1}{\Im z_1}.$$
\end{corollary}

\begin{remark} In case $D=1$, the function $\sum_{a>0}\frac{N_{aD}(m)}{a^2}=\sum_{a>0}\frac{\varphi(a)}{a^2}$, where $\varphi(a)$ is the  Euler function. In this case, the last series equals $\frac{\zeta (2s-1)} {\zeta (2s)}, $ and the zero coefficient equals $\frac{12}{z_1-\bar{z_1}}$.
\end{remark}

\section{The Fourier expansion and the derivatives of the Cauchy kernel}
\emph{Proof of Theorem 1.}

\begin{enumerate}
  \item Note that  \begin{multline} \lim_{s \rightarrow 1}\Xi_{\mathrm{Hil},m}(z, s)(z_2-\bar{z_2})=\lim_{s \rightarrow 1}\sum_{B \in \mathcal{B}_m}\frac{1}{az_1+\lambda'}\left(\frac{\overline{(M(z_1, \bar{z_2}), B)}}{|(M(z_1, \bar{z_2}), B)|^{2s}}-\frac{\overline{(M(z_1, z_2), B)}}{|(M(z_1, z_2), B)|^{2s}}\right)=\\=\varphi_B(z_1, z_2, s)-\varphi_B(z_1, \bar{z_2}, s). \label{differXi} \end{multline}
  \item Let
  \begin{equation}  \varphi_B(z_1, z_2, s) = \sum_{B \in \mathcal{B}_m}\frac{1}{(az_1+\lambda')}\cdot\frac{\overline{(M(z_1, z_2), B)}}{|(M(z_1, z_2), B)|^{2s}}, ~~~     \label{Zagierqz2} \end{equation}
   \begin{equation}  \varphi_B(z_1, \bar{z_2}, s)= \sum_{B \in \mathcal{B}_m}\frac{1}{(az_1+\lambda')}\cdot\frac{\overline{(M(z_1, \bar{z_2}), B)}}{|(M(z_1, \bar{z_2}), B)|^{2s}}.   \label{Zagierqbar2} \end{equation}
    \item    The summation over the set of matrices $\mathcal{B}_m $ can be represented as the sum over $ a \neq 0 $ and $a =0$:  $$\varphi_B(z_1, z_2, s)=\varphi^0_B(z_1, z_2, s)+\sum_{a>0}\varphi^a_B(z_1, z_2, s).$$
   \item \underline{Case 1: $a=0$.}

In such a case
$$\varphi^0_B(z_1, z_2, s) =\sum_{B \in \mathcal{B}_m} \frac{1}{\lambda'}\frac{\lambda \bar{z_1}+\lambda'\bar{z_2}+b}{|\lambda z_1+\lambda'z_2+b|^{2s}}; ~~~~~~\varphi^0_B(z_1, \bar{z_2}, s)=\sum_{B \in \mathcal{B}_m} \frac{1}{\lambda'}\frac{\lambda \bar{z_1}+\lambda'z_2+b}{|\lambda z_1+\lambda'\bar{z_2}+b|^{2s}}.$$
In these series the summation is over $\lambda \in \mathfrak{d}^{-1}_F$, such that $\lambda\lambda'= m/D$, and over $b \in \mathbb {Z}$.

Using the well-known formula
           $$\lim_{s \to 1} \sum_{n
            =-\infty}^{+\infty}\frac{\left(\bar{z}+n\right)}{\left|z+n\right|^{2s}}=-2\pi i\left(\frac{1}{2}+ \sum_{r>0} e^{2 \pi i r z}\right),
           $$
we obtain  \begin{equation}\varphi^0_B(z_1, z_2, s) -\varphi^0_B(z_1, \bar{z_2}, s) =-2 \pi i  \sum_{\substack{\lambda \in \mathfrak{d}^{-1}_F,~\lambda>0 \\ \lambda\lambda'=-m/D}}\sum_{r>0}~\frac{1}{\lambda'}\left( e^{2 \pi i r (\lambda z_1+\lambda' z_2)}-e^{2 \pi i r (\lambda z_1+\lambda' \bar{z_2})}\right) \end{equation}
 \item   \underline{Case 2: $a>0$.}\\
   Let $\mathcal{B}_m=\left\{B=\left(\begin{matrix}
	-b & \lambda'\\
	\lambda & -a\\
	\end{matrix}\right) |~~ \lambda \in \mathfrak{d}^{-1}_F; ~ a, b \in \mathbb{Z},~ \det B = m/ D \right\}.$ The group $\mathrm{SL}_2(F)$ acts on $\mathcal{B}_m$ as follows $$g.B=gBg'^{t}.$$ If $g=g_{\theta}=\left(\begin{matrix}
	1 & \theta\\
	0 & 1\\
	\end{matrix}\right)$, then
 $$g_{\theta}.B=\left(\begin{matrix}
	1 & \theta\\
	0 & 1\\
	\end{matrix}\right)\left(\begin{matrix}
	-b & \lambda'\\
	\lambda & -a\\
	\end{matrix}\right)\left(\begin{matrix}
	1 & 0\\
	\theta ' & 1\\
	\end{matrix}\right)=\left(\begin{matrix}
	-b+\theta\lambda+ \theta'\lambda' -a\theta\theta' & \lambda'-a\theta\\
	\lambda-a\theta' & -a\\
	\end{matrix}\right).$$
Conversely, if $B_1=\left(\begin{matrix}
	-b_1 & \lambda'_1\\
	\lambda_1 & -a_1\\
	\end{matrix}\right) \in \mathcal{B}_m$, where
 $\lambda_1 \equiv \lambda_0 \pmod{a \mathcal{O}_F}$, $\lambda_1=\lambda_0-a\theta'$, $\lambda_1'=\lambda_0'-a\theta$, since $a_1b_1-\lambda_1\lambda_1'=m/D$, then $b_1=\dfrac{\lambda_1\lambda_1'+m/D}{a}=\dfrac{\lambda_0\lambda_0'-a\lambda\theta-a\lambda'\theta'+a^2\theta\theta'+m/D}{a}=b_0-\lambda_0\theta-\lambda_0'\theta'-a\theta\theta',$ and  $B_1=g_{\theta}.B_0$, where $B_0=\left(\begin{matrix}
	-b_0 & \lambda'_0\\
	\lambda_0 & -a_0\\
	\end{matrix}\right).$ Consequently, for a fixed $ a> 0$, we can replace the summation over the matrices from $\mathcal {B}_m $ by the summation over the all residues $\mathcal{R}= \left\{\lambda_0 \in \mathfrak{d}^{-1} \pmod{a\mathcal{O}_F}|~ \mathrm{N}(\lambda_0 \sqrt{D})=-\lambda_0\lambda_0'D \equiv m \pmod{aD}\right\}$ and all $\theta \in \mathcal{O}_F$: $$\varphi^a_B(z_1, z_2, s)=\sum_{\lambda_0 \in \mathcal{R}} \sum_{\theta \in \mathcal{O}_F} \widetilde{\varphi}^a_{g_{\theta}\cdot B_0}(z_1, z_2, s).$$

Since,

 $(M(z), B)=\mu_{\gamma}(z_1, z_2)=az_1z_2+\lambda_1z_1+\lambda_1'z_2+b=a(z_1-\theta)(z_2-\theta')+\lambda_0(z_1-\theta)+\lambda_0'(z_2-\theta')+b_0$,

where $b_0=\dfrac{\lambda_0\lambda_0'+m/D}{a}$, then, for $a>0$,  (\ref{Zagierqz2}) rewrite as follows:
\begin{multline} 2 \sum_{a>0} \varphi_B^a(z_1, z_2, s) = \\= 2\sum_{a>0}\sum_{\lambda_0 \in \mathcal{R}} \sum_{\theta \in \mathcal{O}_F} \frac{1}{a(z_1-\theta)+\lambda_0'}\cdot \frac{a(\bar{z_1}-\theta)(\bar{z_2}-\theta')+\lambda_0(\bar{z_1}-\theta)+\lambda_0'(\bar{z_2}-\theta')+b_0}{|a(z_1-\theta)(z_2-\theta')+\lambda_0(z_1-\theta)+\lambda_0'(z_2-\theta')+b_0|^{2s}}.  \label{Zagier3} \end{multline}
Notice, that \begin{multline}a(z_1-\theta)(z_2-\theta')+\lambda_0(z_1-\theta)+\lambda_0'(z_2-\theta')+\dfrac{\lambda_0\lambda_0'-m/D}{a}=\\=a\left[\left( z_1+\frac{\lambda_0'}{a}-\theta\right)\left( z_2+\frac{\lambda_0}{a}-\theta'\right)-\frac{m}{a^2D} \right].\end{multline} Let $\dfrac{m}{a^2D}=\Delta$. Therefore, the series (\ref {Zagier3}) equals to
\begin{multline}\sum_{a>0}\varphi^a_B(z_1, z_2, s)=\sum_{a>0}\sum_{\lambda_0 \in  \mathcal{R}}\sum_{\theta \in  \mathcal{O}_F} \frac{1}{a^{2s}}\cdot \frac{1}{(z_1+\lambda_0'/a-\theta)} \cdot \frac{(\bar{z_1}+\lambda_0'/a-\theta)(\bar{z_2}+\lambda_0'/a-\theta')-\Delta}{|(z_1+\lambda_0'/a-\theta)(z_2+\lambda_0'/a-\theta')-\Delta|^{2s}} =\\
=\sum_{\nu \in \mathfrak{d}^{-1}}\left(\sum_{a>0}\sum_{\lambda_0 \in \mathcal{R}} \frac{e^{2 \pi i\left(\nu\frac{\lambda_0'}{a}+\nu'\frac{\lambda_0}{a}\right)}}{a^{2s}} \cdot b_{\nu, s}(y_1, y_2, \Delta)e^{2 \pi i (\nu z_1+\nu'z_2)}   \right), \label{varphi1} \end{multline}
where $b_{\nu, s}( y_1, y_2, \Delta)$ is the Fourier coefficients of the function $$\sum_{\lambda_0 \in \mathcal{R}}\frac{1}{(z_1-\theta)}\cdot\frac{(\bar{z_1}-\theta)(\bar{z_2}-\theta')-\Delta}{|(z_1-\theta)(z_2-\theta')-\Delta|^{2s}}=\sum_{\nu \in \mathfrak{d}^{-1}} b_{\nu, s}(y_1, y_2, \Delta)e^{2\pi i (\nu z_1+\nu'z_2)},$$ and  $y_1=\Im z_1,~ y_2=\Im z_2, ~ y_1y_2 > m/D$.

Similarly, for the series (\ref{Zagierqbar2}) at $a> 0$:
\begin{multline}\sum_{a>0}\varphi^a_B(z_1, \bar{z_2}, s)=\sum_{\nu \in \mathfrak{d}^{-1}}\left( \sum_{a>0}\sum_{\lambda_0 \in \mathcal{R}} \frac{e^{2 \pi i\left(\nu\frac{\lambda_0'}{a}+\nu'\frac{\lambda_0}{a}\right)}}{a^{2s}} \cdot \widetilde{b_{\nu, s}}(y_1, -y_2, \Delta)e^{2 \pi i (\nu z_1+\nu'\bar{z_2})}   \right). \label{varphi2} \end{multline}
\item Now we find the coefficients $ b_{\nu, s} (y_1, y_2, \Delta) $ and $\widetilde {b_{\ nu, s}} (y_1, -y_2, \Delta) $ using the Poisson summation formula.
\begin{multline}  b_{\nu, s}(y_1, y_1, \Delta)=\frac{1}{\sqrt{D}}\int_{\Im z_1=y_1}\int_{\Im z_2=y_2} e^{-2 \pi i (\nu z_1+\nu'z_2)} \frac{1}{z_1} \cdot \frac{\overline{z_1z_2}-\Delta}{|z_1z_2-\Delta|^{2s}}~dz_1~dz_2=\\ = \frac{1}{\sqrt{D}}\int_{\Im z_1=y_1} \frac{\bar{z_1}^2}{|z_1|^{2s+2}}e^{-2 \pi i (\nu z_1+\nu'\Delta/z_1)}\int_{\Im z_2=y_2} \frac{\overline{z_2}-\Delta/ \bar{z_1}}{|z_2-\Delta/ z_1|^{2s}} e^{-2 \pi i (\nu' z_2-\nu'\Delta/z_1)} ~dz_2~dz_1 =\\
= \frac{1}{\sqrt{D}}\int_{\Im z_1=y_1} \frac{\bar{z_1}^2}{|z_1|^{2s+2}}e^{-2 \pi i (\nu z_1+\nu'\Delta/z_1)}\int_{\Im t=y_2-\Im (\Delta/ z_1)} \frac{\bar{t}}{|t|^{2s}} e^{-2 \pi i \nu' t} ~dt~dz_1,  \label{IntConv1}
\end{multline}
where $t=z_2- \Delta/ z_1$. The inner integral has the form:
\begin{multline} \int_{\Im t=C} \frac{\bar{t}}{|t|^{2s}}e^{-2\pi i \nu' t} ~dt = \int_{-\infty}^{+\infty}e^{-2\pi i \nu(x+iy)}\frac{x-iy}{|x+iy|^{2s}}~dx =\\=e^{2\pi \nu y}\int_{-\infty}^{+\infty}\frac{e^{-2 \pi i \nu x}(x-iy)}{(x^2+y^2)^s}~dx=y^{2s-1}e^{2\pi \nu y}\int_{-\infty}^{+\infty}\frac{e^{-2\pi i \nu y u}(u-i)}{(u^2+1)^s}~du,  \label{IntConv2}
\end{multline}
where $u=\dfrac{x}{y}$,~ $du =\dfrac{1}{y}~ dx$.\\

The integrand $\dfrac{e^{-2\pi i \nu uy}(u-i)}{(u^2+1)^s}$ in the last integral is the one-valued analytic function in the cut region $\mathbb{C}-\left[i, i \infty \right] - \left[-i\infty, -i\right] $. It can be shown in the usual way that the last integral  converges uniformly with respect to the parameter $s$ for $\Re(s) \geq 1$ and holomorphic in $s$. Moreover, the series \begin{equation} \sum_{\nu \in \mathfrak{d}^{-1}} b_{\nu, s}(y_1, y_2, \Delta)e^{2\pi i (\nu z_1+\nu'z_2)},\label{SumP(r)} \end{equation} also absolutely convergent for $\Re(s) \geq 1$, which follows from the uniform estimate \cite{Ku} of the integral and Zagier estimate (\ref{est}). This implies that the series (\ref{SumP(r)}) has a limit at $s=1$, which can be obtained by setting $s=1$ in each term. Therefore,
\begin{multline}  b_{\nu, 1}(y_1, y_2, \Delta) =\frac{1}{\sqrt{D}}\int_{-\infty+iy_1}^{+ \infty+iy_1} \frac{1}{z_1^{2}}e^{-2 \pi i (\nu z_1+\nu'\Delta/z_1)}\int_{-\infty+iy_2}^{+ \infty+iy_2} \frac{1}{z_2-\Delta/z_1} e^{-2 \pi i \nu'(z_2-\Delta/z_1)} ~dz_2~dz_1.  \label{bnu}
\end{multline}

\item The inner integral in the last expression, $$\int_{-\infty+iy_2}^{+ \infty+iy_2} \frac{1}{z_2-\Delta/z_1}e^{-2 \pi i \nu'(z_2-\Delta/z_1)} ~dz_2,$$ has a pole in $z_2-\Delta/z_1=0$. Since $\Delta >0$, $\Im z_1>0$ and $\Im z_2 \Im z_1 > m/D$, then, for $\nu '<0 $, we can deform the path of integration to $i \infty $ without crossing the poles, so the inner integral is zero. Consequently, in the Fourier expansion of the function $\varphi^a_B (z_1, z_2, s)$ there are no terms with $\nu '<0$. \\

From analogous calculations, it follows that in the Fourier expansion of the series ~$\varphi^a_B(z_1, z_2, s)$~ there are no terms with $\nu, \nu' <0$, and, in the expansion of the series ~$\varphi^a_B(z_1, \bar{z_2}, s)$~ there are no terms with $\nu<0$  è $\nu' >0$.

\item \textbf{Case $\nu, \nu' >0$}.

 Now consider the inner integral in the expression (\ref{bnu}). Let $t=z_2-\Delta/z_1$, then  $$\int_{-\infty+iy_2}^{+ \infty+iy_2} \frac{1}{z_2-\Delta/z_1} e^{-2 \pi i \nu'(z_2-\Delta/z_1)} ~dz_2=\int_{-\infty+iy}^{+\infty + iy}\frac{e^{-2\pi i \nu' t}}{t} dt = -2\pi i~ \mathrm{res}_{t=0}\left(\frac{e^{-2\pi i \nu' t}}{t}\right)=-2 \pi i.$$ Hence, \begin{equation} b_{\nu, 1}(y_1, y_2, \Delta)=\frac{-2\pi i}{\sqrt{D}}\int_{-\infty+iy_1}^{+ \infty+iy_1}\frac{1}{z_1^2}e^{- 2\pi i (\nu z_1+\frac{\nu'\Delta}{z_1})} dz_1.  \label{blast} \end{equation}
Let $t=-2\pi i \nu z_1$, $\alpha=4\pi \sqrt{\nu\nu'\Delta}$, then, in case $\nu, \nu' >0$, we obtain
 \begin{equation}b_{\nu}(y_1, y_2, \Delta)=-\frac{(2\pi)^2 i}{\sqrt{D}} \sqrt{\frac{\nu}{\nu'\Delta}}~J_1(4\pi \sqrt{\nu\nu'\Delta})=-4\pi^2 i ~a \sqrt{\frac{\nu}{\nu'm}}~I_1\left(\frac{4\pi \sqrt{\nu\nu'm/D}}{a}\right).
  \end{equation}
For the series $\varphi^a_B(z_1, \bar{z_2}, s)$, in case $\nu>0$  and $\nu' <0$, by the same way we obtain \begin{equation}\widetilde{b_{\nu, 1}}(y_1, -y_2, \Delta)=-4\pi^2 i~ a \sqrt{\frac{\nu}{\nu'm}}~J_1\left(\frac{4\pi \sqrt{\nu\nu'm/D}}{a}\right).
  \end{equation}
\item \textbf{Case $\nu, \nu'=0$}.
\begin{multline} b_{0, s}(y_1, y_2, \Delta)=\frac{1}{\sqrt{D}} \int_{\Im z_2=y_2}\int_{\Im z_1=y_1} \frac{\overline{z_1z_2}-\Delta}{z_1|z_1z_2-\Delta|^{2s}}~dz_1~dz_2=\\
= \frac{1}{\sqrt{D}}\int_{\Im z_2=y_2} \frac{\bar{z_2}}{|z_2|^{2}}\int_{\Im z_2=y_2} \frac{\bar{z_1}-\Delta/\bar{z_2}}{z_1|z_1-\Delta/z_2|^{2s}} ~dz_1 dz_2. \label{bnuo}
\end{multline}
Let $z_i=x_i+iy_i$. The inner integral
\begin{equation} \int_{\Im z_1=y_1}\frac{x_1-iy_1-\frac{\Delta(x_2+iy_2)}{x_2^2+y_2^2}}{(x_1+iy_1)\left|x_1+iy_1-\frac{\Delta(x_2-iy_2)}{x_2^2+y_2^2}\right|^{2s}}~dx_1.
\end{equation}
 If $\alpha=-\dfrac{\Delta x_2}{x_2^2+y_2^2}$, $\beta=\dfrac{\Delta y_2}{x_2^2+y_2^2}$, $u=\dfrac{x_1+\alpha}{y_1+\beta},$ then for $s \rightarrow 1$  the last integral equals to \begin{equation}
\frac{1}{(y_1+\beta)^{2s-1}} \int_{-\infty}^{+\infty}\frac{(u-i)}{(u+i)(u^2+1)^s} \left[\sum_{k=0}^{\infty}\left(\frac{\alpha+i\beta}{y_1+\beta}\right)^k\cdot\frac{1}{(u+i)^k} \right]~du. \label{intsum} \end{equation}
There is the following representation for the third Cauchy beta-integral: $$c(x, y)= \int_{-\infty}^{+\infty} \frac{dt}{(1+it)^x(1-it)^y}=\frac{\pi 2^{2-x-y} \Gamma(x+y-1)}{\Gamma(x) \Gamma(y)}.$$
Using the third beta-integral, we obtain that the integrals in the sum (\ref{intsum}) have the form
\begin{equation}  \int_{-\infty}^{+\infty}\frac{(u-i) ~du}{(u+i)^{k+1}(u^2+1)^s}=\frac{\pi~ 2^{2-k-2s}\Gamma(2s +k-1)}{\Gamma(s+k+1)\Gamma(s-1)}\cdot\frac{1}{i^{k+2}}.   \end{equation}
Since $\frac{\alpha+i\beta}{y_1+\beta}=\frac{x_2-iy_2}{\frac{a^2D}{m}(x_2^2+y_2^2)+y_2}, $ then the inner integral in (\ref{bnuo}) equals to the sum:
\begin{equation} \frac{1}{(y_1+\beta)^{2s-1}}\sum_{k=0}^{\infty}\left(\frac{x_2-iy_2}{\frac{a^2D}{m}(x_2^2+y_2^2)+y_2}\right)^k\frac{\pi 2^{-2s-k+ 2}\Gamma(2s+k+1)}{\Gamma(s+k+1)\Gamma(s-1)}.  \end{equation}
Let \begin{equation} c_k(s)=\frac{\pi 2^{-2s-k+ 2}\Gamma(2s+k+1)}{\Gamma(s+k+1)\Gamma(s-1)}.
\end{equation} Note that all $c_k(s) $ have the first-order zero at $s=1$.
Returning to the double integral (\ref{bnuo}) and to the coefficient $b^0_{\nu, s}(y_1, y_2, \Delta)$, we get
\begin{multline} b^0_{\nu, s}(y_1, y_2, \Delta)=  \frac{1}{\sqrt{D}}\int_{\Im z_2=y_2}\frac{\bar{z_2}}{|z_2|^{2}} \frac{1}{\left(y_1+\frac{\Delta}{x_2^2+y_2^2}\right)^{2s-1}}\sum_{k=0}c_k(s)\left(\frac{x_2-iy_2}{\frac{a^2D}{m}(x_2^2+y_2^2)+y_2}\right)^k ~dz_2= \\ = \sum_{k=1}^{\infty}c_k(s) I_k(y_1, y_2). \end{multline}

(1) Let $u=x_2/y_2$. The integrals in the last sum have the form: \begin{multline} I_k(y_1, y_2)=\int_{\Im z_2=y_2}\frac{\bar{z_2}}{|z_2|^{2}} \frac{1}{\left(y_1+\frac{\Delta}{x_2^2+y_2^2}\right)^{2s-1}}\frac{(x_2-iy_2)^k}{\left(\frac{a^2D}{m}(x_2^2+y_2^2)+y_2\right)^k}~dz_2=\\=\left(\frac{\Delta}{y_1y_2}\right)^k\cdot\frac{1}{y_2^{2s-2}y_1^{2s-1}} \int_{-\infty}^{+\infty} \frac{(u-i)^{k+1}~du}{(u^2+1)^s\left(u^2+1+\frac{\Delta}{y_1y_2}\right)^{k}\left(1+\frac{\Delta/(y_1y_2)}{u^2+1}\right)^{2s-1}}.
\end{multline}
So, for $s \rightarrow 1$ we have:
\begin{multline} \lim_{s\rightarrow 1} I_k(y_1, y_2)= \left(\frac{\Delta}{y_1y_2}\right)^k\int_{-\infty}^{+\infty}\frac{(u-i)^{k+1} du}{(u^2+1)(u^2+1+\Delta/(y_1y_2))^k\left(1+\frac{\Delta/(y_1y_2)}{ (u^2+1)} \right)} =\\
= \int_{-\infty}^{+\infty}\frac{~du}{(u+i)^{k+1}} ~\sum_{n=0}^{\infty}(-1)^{n+k} \frac{(n+k)!}{n!} \left( \frac{\Delta/ y_1y_2}{u^2+1}\right)^n.\end{multline}

 Using the third Cauchy beta-integral, we obtain $$I_k(y_1, y_2)=\frac{2^{1-k-2n}\pi}{i^{k+1}}\sum_{n=0}^{\infty}\left(\frac{\Delta}{y_1y_2}\right)^n \frac{\Gamma(k+2n)}{\Gamma(k+n+1)\Gamma(n)}.$$

(2) Since $b_{0,1}(y_1, y_2, \Delta)=\frac{1}{\sqrt{D}}\sum_{k=0} c_k(s)I_k(y_1, y_2),$
then, substituting the zero coefficient ($\nu = 0$) in the Fourier expansion for $\varphi^a_B (z_1, z_2, s) $, we get:
\begin{equation} \sum_{a>0} \frac{G_a(m, o)}{a^{2s}} b_{0, 1}(y_1, y_2, \Delta)=\sum_{k>0}\frac{1}{\sqrt{D}}\frac{G_a(m, 0)}{a^{2s}} c_k(s)~ \frac{\pi}{i^{k+1}}\left[I_0(y_1, y_2)+\frac{1}{a^2}\left[~\ldots~\right]\right].
\end{equation}
For $ s = 1$ the series $\sum_{a> 0} \frac{G_a (m, o)}{a^{2s + 2}} $ is the holomorphic function, and all $c_k(s)$ have the first-order zero, 	
therefore all the terms in this sum, except the first, vanish. Hence,
\begin{multline}\lim_{s \rightarrow 1}~\sum_{a>0}\frac{G_a(m, 0)}{a^{2s}} b_{0, 1}(y_1, y_2, \Delta) = \lim_{s \rightarrow 1} \frac{1}{\sqrt{D}}\sum_{a>0}\frac{G_a(m, 0)}{a^{2s}}c_1(s)\cdot \int_{-\infty}^{+\infty} \frac{\bar{u}+i}{(u^2+1)^s} du =\\= \lim_{s \rightarrow 1} \frac{2\pi^2  }{\sqrt{D}}\sum_{a>0}\frac{G_a(m, 0)}{a^{2s}}\cdot\frac{1}{\Gamma(s-1)} (z_1-\bar{z_1})^{-1}.
\end{multline}

For the case (\ref{Zagierqbar2}):$$b_{0, 1}(y_1, \bar{y_2}, \Delta)=-\lim_{s \rightarrow 1} \frac{2\pi^2 }{\sqrt{D}}~\frac{1}{\Gamma(s-1)} (z_1-\bar{z_1})^{-1}.$$
\end{enumerate}
This completes the proof.
\begin{flushright}
  $\blacksquare$
\end{flushright}

Now we will calculate the derivatives of the function $(z_2-\bar{z_2})^{2s-1}\Xi_{\mathrm{Hil},m}(z_1,z_2, s)$ for $n=1$ and $\Re(s)>1$.

In order to define the Zagier series \cite{Za1} $$\omega_m(z_1,\bar{z_2}, k)= \frac{1}{2}\sum\limits_{\substack{\gamma=\left(\begin{smallmatrix}
	\lambda & b\\
	a & \lambda'\\
	\end{smallmatrix}\right) \\ab - N(\lambda) =m/D}} \dfrac{1}{\mu_{\gamma}(z_1, -\bar{z_2})^k}$$ for $k=2$, let us consider the series
 	\begin{equation}
\Omega_{m}(z_1,\bar{z_2},s)=\frac{1}{2}\sum \limits_{\substack{\gamma=\left(\begin{smallmatrix}
	\lambda & b\\
	a & \lambda'\\
	\end{smallmatrix}\right) \\ab - N(\lambda) =m/D}}\frac{\mu_{\gamma}(\bar{z_1},-z_2)^{2}}{|\mu_{\gamma}(z_1, -z_2)|^{2s-2}|\mu_{\gamma}(z_1, -\bar{z_2})|^{2s+2}}.
		\end{equation}
One can prove that the series $\Omega_{m}(z_1,\bar{z_2},s)$ can be analytically continued to the point~$s=1$.

Therefore, we can put $\omega_m(z_1,\bar{z_2})=\lim\limits_{s\rightarrow 1}\Omega_{m}(z_1,\bar{z_2},s)$.

\begin{lemma}
$$\overline{\partial } ~\Xi_{\mathrm{Hil}, m}(z_1,z_2)(z_2-\bar{z_2}) = -\lim_{s \rightarrow 1} \frac{2\pi^2  }{\sqrt{D}}\sum_{a>0}\frac{G_a(m, 0)}{a^{2s}}\cdot\frac{1}{\Gamma(s-1)}\frac{1}{(z_1-\bar{z_1})^2}-\omega_{ m}(z_1, \bar{z_2}) ~d\bar{z_2}.$$
\end{lemma}
\begin{proof}
 Differentiation of the function $(z_2-\bar{z_2})^{2s-1}\Xi_{\mathrm{Hil}, m}(z_1,z_2, s)$ with respect to $\bar{z_1}$ gives
  \begin{multline}
\lim\limits_{s \rightarrow 1}\frac{d}{d\bar{z_1}}(z_2-\bar{z_2})^{2s-1}\Xi_{\mathrm{Hil}, m}(z_1,z_2,s)=(s-1)\frac{(z_2-\bar{z_2})^{2s-1}}{z_1-\bar{z_1}} \times [ 2~\Xi_{\mathrm{Hil}, m}(z_1,z_2,s)-\\
- \frac{1}{2} \sum_{\substack{\gamma=\left(\begin{smallmatrix}
	\lambda & b\\
	a & \lambda'\\
	\end{smallmatrix}\right) \\ab - N(\lambda) =m/D}}\frac{1}{|\mu_{\gamma}(z_1, -z_2)|^{2s}|\mu_{\gamma}(z_1, -\bar{z_2})|^{2s-2}}- \frac{1}{2} \sum_{\substack{\gamma=\left(\begin{smallmatrix}
	\lambda & b\\
	a & \lambda'\\
	\end{smallmatrix}\right) \\ab - N(\lambda) =m/D}}\frac{1}{|\mu_{\gamma}(z_1, -z_2)|^{2s-2}|\mu_{\gamma}(z_1, -\bar{z_2})|^{2s}}]=\\=-\lim_{s \rightarrow 1} \frac{2\pi^2  }{\sqrt{D}}\sum_{a>0}\frac{G_a(m, 0)}{a^{2s}}\cdot\frac{1}{\Gamma(s-1)}\frac{1}{(z_1-\bar{z_1})^2}.
\end{multline}

The partial derivative with respect to $\bar{z_2}$ gives
  \begin{multline}
\frac{d}{d\bar{z_2}}~\Xi_{\mathrm{Hil}, m}(z_1,z_2,s)(z_2-\bar{z_2})^{2s-1}=\\
=(z_2-\bar{z_2})^{2s-2}[(1-s)\sum_{\substack{\gamma=\left(\begin{smallmatrix}
	\lambda & b\\
	a & \lambda'\\
	\end{smallmatrix}\right) \\ab - N(\lambda) =m/D}}\frac{\mu_{\gamma}(\bar{z_1},-z_2)^{2}}{|\mu_{\gamma}(z_1, -z_2)|^{2s}|\mu_{\gamma}(z_1, -\bar{z_2})|^{2s}}- s~\Omega_1(z_1, \bar{z_2},s) ].
\end{multline}
It follows that
\begin{equation}
\lim_{s \rightarrow 1} \frac{d}{d\bar{z_2}}~\Xi_{\mathrm{Hil}, m}(z_1,z_2,s)(z_2-\bar{z_2})^{2s-1}=-\frac{1}{2}\sum_{\substack{\gamma=\left(\begin{smallmatrix}
	\lambda & b\\
	a & \lambda'\\
	\end{smallmatrix}\right) \\ab - N(\lambda) =m/D}}\frac{1}{\mu_{\gamma}(z_1, \bar{z_2})^2}=-\omega_m(z_1, \bar{z_2}). \label{Dz_2Xi}
\end{equation}

\end{proof}

\newpage


\bibliographystyle{amsplain}

\end{document}